\documentclass[12pt,a4paper,reqno]{article}
\usepackage{amsmath}
\usepackage{amsthm}
\usepackage{enumerate}
\usepackage{amssymb}

\usepackage{verbatim}
\usepackage{url}
\usepackage[latin1]{inputenc}

\usepackage{caption}
\usepackage{graphicx,color}
\def\N{{\mathbb N}}
\def\Z{{\mathbb Z}}

\def\R{{\mathbb R}}

\def\Hb{{\mathbb H}}

\def\Ec{{\mathcal E}}

\def\Hc{{\mathcal H}}

\def\Wc{{\mathcal W}}

\def\ebf{\mathbf{e}}

\def\Flow{F_\textup{sub}}
\def\mulow{\mu_\textup{sub}}

\def\Fhi{F_\textup{dom}}

\def\Yhi{Y_\textup{dom}}
\def\Cyl{\mathcal{C}}
\def\tbar{\bar{t}}

\def\be{\begin{equation}}
\def\ee{\end{equation}}
\def\bea{\begin{equation*}}
\def\eea{\end{equation*}}

\def\eps{\varepsilon}

\def\Pbf{{\bf P}}

\def\Pr{{\mathbb P}}
\DeclareMathOperator{\E}{{\mathbb E}}

\DeclareMathOperator{\supp}{supp}

\newtheorem{thm}{Theorem}[section]
\newtheorem{lma}[thm]{Lemma}

\newtheorem{prop}[thm]{Proposition}

\newtheorem{claim}{Claim}
\newtheorem*{claim*}{Claim}

\theoremstyle{remark}
\newtheorem{remark}[thm]{Remark}
\newtheorem{preex}[thm]{Example}

\theoremstyle{definition}

\numberwithin{equation}{section}

\begin{document}

\title{Inhomogeneous first-passage percolation}
\date{\today}
\author{Daniel Ahlberg \\ \small IMPA \and Michael Damron \\ \small Indiana University, Bloomington\\ \small Princeton University \and Vladas Sidoravicius \\ \small IMPA}
\maketitle

\begin{abstract}
We study first-passage percolation where edges in the left and right half-planes are assigned values according to different distributions. We show that the asymptotic growth of the resulting inhomogeneous first-passage process obeys a shape theorem, and we express the limiting shape in terms of the limiting shapes for the homogeneous processes for the two weight distributions. We further show that there exist pairs of distributions for which the rate of growth in the vertical direction is strictly larger than the rate of growth of the homogeneous process with either of the two distributions, and that this corresponds to the creation of a defect along the vertical axis in the form of a `pyramid'.
\end{abstract}

\section{Introduction}

First-passage percolation is a stochastic model for spatial growth that has been widely studied by mathematicians and physicists (see e.g.~\cite{kruspo91,kesten03,howard04}). Since the pioneering work of Eden~\cite{eden58}, both communities have benefited from intense activity connected to first-passage percolation, resulting in a rigorous theory for subadditive ergodic processes~\cite{kingman73}, and far reaching predictions of KPZ-theory~\cite{karparzha86}. Typically one assigns nonnegative i.i.d.\ weights to the edges of the usual integer lattice in two or more dimensions, and studies the pseudo-metric $T$ induced by the resulting weighted graph. In this paper we introduce an inhomogeneous version in which edges in the left and right half-planes of the $\Z^2$ lattice are assigned weights according to distributions $F_-$ and $F_+$. The first fundamental question is then whether the Shape Theorem still holds; that is, does the rescaled ball $B_t=\{x:T(0,x)\le t\}$ obey a law of large numbers as in the usual 
model, and if so, what does the limiting shape for $B_t/t$ 
look like?

Unlike in the homogeneous case, when $F_-=F_+$, we cannot rely on the usual ergodic theory for subadditive processes due to the lack of horizontal translation invariance. To establish existence of radial limits, a precursor to the Shape Theorem, we instead complement the classical approach using large deviation estimates for half-plane passage times introduced in~\cite{A13-2}.
Our method also shows how the asymptotic shape of $B_t/t$ can be described in terms of the shapes for homogeneous first-passage percolation with $F_-$ or $F_+$. If either of $F_-$ and $F_+$ dominates the other (in a concave stochastic ordering), then the asymptotic shape is the convex hull of the restriction to respective half-planes of the asymptotic shapes for $F_-$ and $F_+$. When no such relation is present, the asymptotic shape equals the convex hull of the two half-shapes and a potentially wider
additional line segment along the vertical axis
(see Figure~\ref{fig: shape_thm}, page~\pageref{fig: shape_thm}).


The behavior of first-passage percolation along the boundary of two regions with different passage times has attracted much attention due to its physical relevance and mathematical challenge. In~\cite{vdbkes93} it was shown that if passage times with distribution $F$ are replaced throughout the graph by others, distributed according to $F'$, which is strictly smaller than $F$ in a stochastic sense, then the time constant changes strictly. But how would the passage time change, for example, in the first coordinate direction $\mathbf{e}_1$ if passage times are modified only along edges lying on the $\mathbf{e}_1$-axis?  Will an arbitrarily small modification be detectable on a macroscopic scale? This is the well-known `columnar defect' problem, which has been studied in many forms both numerically and rigorously; see for instance~\cite{woltan90,kanmuk92,mylmaumertimhorhanij03}. There is no satisfactory theory which would explain why some models are `sensitive'  to any arbitrarily small perturbation and others 
are not; this is determined by competition between localized reinforcement, induced by an impurity, and bulk fluctuations, which in many cases are difficult to analyze. Perhaps the most prominent example is the one-dimensional totally asymmetric simple exclusion process with a `slow bond' at the origin and particle density $\rho = 1/2$~\cite{janleb92,janleb94}, which for some initial conditions can be represented either in terms of last-passage percolation with a columnar defect or the so-called Poly-Nuclear Growth model ({PNG}), for which it is believed that any perturbation will be reflected at the macroscopic level in the change of the current. In line with the terminology of this article, this would mean the creation of a defect in form of a `pyramid'. A theory of propagation of influence of impurities was proposed in~\cite{befsidspovar06} and later used to show that for a randomized PNG model with a columnar defect, only modifications above a certain threshold result in a `spike' on the macroscopic 
profile~\cite{befsidvar10}.


The inhomogeneous model introduced in this paper is rich enough to display defects at a macroscopic scale: We show in Theorem~\ref{thm: strict_inequality} that there are pairs of distributions $(F_-,F_+)$ for which the speed of growth along the vertical axis is faster than it would be for homogeneous first-passage percolation with either of $F_-$ and $F_+$. Due to subadditivity in the model, the enhanced speed in the vertical direction does not create a `spike' on the limiting shape (as in randomized polynuclear growth), but instead a `pyramid' (see\ Theorem~\ref{thm: mu_bar_properties}). As in pinning phenomena for polymers, the formation of a pyramid indicates that minimizing paths in the vertical direction are `attracted' to the vertical axis, benefiting from low-weight edges in both half-planes. Moreover, our approach extends to the more general situation where edges in the left and right half-planes are assigned weights according to $F_-$ and $F_+$ respectively, but edges on the vertical axis are 
assigned weights according to a third distribution $F_0$ (see Remark~\ref{rem: three_f}). This covers, in particular, the above mentioned case of a columnar defect.

We conclude by mentioning two open problems. First, for which pairs of distributions $(F_-,F_+)$ is a pyramid formed on the limiting shape in the vertical direction? Our theorems give necessary conditions but not sufficient ones. Second, in the case of a columnar defect, meaning $F_+=F_-$, which distributions $F_0$ result in a pyramid on the limiting shape? Again, we know only of necessary conditions. In Section~\ref{sec: random_defects} we introduce a related construction, with defects appearing in each column independently at random. For this model we show (see Theorem~\ref{thm: random_defects}) that at all intensities, the contribution of the defects to the time constant prevails, resulting in a strict change.

\subsection{Convergence towards an asymptotic shape}

Let $F_-$ and $F_+$ denote distribution functions of two probability measures supported on $[0,\infty)$. For each edge $e$ in the set of nearest-neighbour edges $\Ec$ of $\Z^2$, assign a random variable $\tau_e$, according to $F_-$ if at least one endpoint of $e$ lies within the left half-plane (contained in $\{(x_1,x_2)\in\R^2:x_1<0\}$), and according to $F_+$ otherwise. For each nearest-neighbour path $\Gamma$ in $\Z^2$ we let $T(\Gamma):=\sum_{e\in\Gamma}\tau_e$. Distances in the induced random psuedo-metric we refer to as \emph{passage times}, defined for $u,v\in\Z^2$ as
\begin{equation*}
T(u,v):=\inf\{T(\Gamma):\Gamma\text{ is a path from $u$ to $v$}\}\ .
\end{equation*}
To describe our results on radial convergence and the shape theorem we introduce two variables:
\begin{equation*}
Y_- = \min\{\tau_-^{(1)}, \ldots, \tau_-^{(4)}\}\quad \text{and}\quad Y_+ = \min\{\tau_+^{(1)}, \ldots, \tau_+^{(4)}\}\ ,
\end{equation*}
where the $\tau_-^{(i)}$'s are i.i.d.\ with distribution $F_-$ and the $\tau_+^{(i)}$'s are i.i.d.\ with distribution $F_+$.
Extend the function $T(x,y)$ to the full space $\R^2$ by identifying $T(x,y)=T(x',y')$ when $x,y\in \mathbb{R}^2$ and $x',y'\in\Z^2$ satisfy $x\in x'+[0,1)^2$ and $y\in y'+[0,1)^2$.

\begin{thm}\label{thm: radial_convergence}
For any $F_-$ and $F_+$ with $\E Y_- <\infty$ and $\E Y_+<\infty$, there exists a function $\bar \mu : \mathbb{R}^2 \to [0,\infty)$ such that for each $x \in \mathbb{R}^2$,
\begin{equation}\label{eq: mu_bar_def}
\frac{T(0,nx)}{n} \to \bar \mu(x)\quad \text{almost surely and in } L^1.
\end{equation}
\end{thm}

The shape theorem that we will prove describes how the set $\Wc_t\subset\R^2$, given by
\begin{equation*}
\Wc_t:=\{x\in\R^2:T(0,x)\le t\}\ ,
\end{equation*} 
compares asymptotically to the set
\begin{equation*}
\overline\Wc:=\{x\in\R^2:\bar\mu(x)\le1\}\ .
\end{equation*}
We write $|\cdot|$ for Euclidean distance on $\R^2$.

\begin{thm}\label{thm: shape_theorem}
Assuming that $\E Y_-^2  < \infty$ and $\E Y_+^2<\infty$,
\begin{equation}\label{eq: inverted_shape}
\limsup_{z\in\Z^2:\,|z| \to \infty} \frac{|T(0,z)-\bar{\mu}(z)|}{|z|} = 0 \quad\text{almost surely}\ .
\end{equation}
If, in addition, $\max\{F_-(0),F_+(0)\}<p_c$, the critical probability for bond percolation on $\mathbb{Z}^2$, then the set $\overline\Wc$ is convex and compact with non-empty interior, and for every $\eps>0$, almost surely,
$$
(1-\eps)\overline\Wc\subset\tfrac{1}{t}\Wc_t\subset(1+\eps)\overline\Wc\quad\text{for large enough }t\ .
$$
\end{thm}

\begin{remark}\label{rem: low_moment}
As in the homogeneous case, if either $Y_-$ or $Y_+$ has infinite mean, then the above almost sure and $L^1$-convergence in~\eqref{eq: mu_bar_def} fails for all $x$ in the interior of the respective half-plane. However, for arbitrary $F_+$ and $F_-$, the convergence in~\eqref{eq: mu_bar_def} holds in probability. A similar weakening holds also for~\eqref{eq: inverted_shape}: For arbitrary $F_-$ and $F_+$
\[
\limsup_{z\in\Z^2:\,|z|\to\infty}\Pr\big(|T(0,z)-\bar\mu(z)|>\eps|z|\big)=0\ .
\]
We omit the argument, but mention that a proof would follow along the lines of Cox and Durrett~\cite[Theorem~1]{coxdur81}. One defines approximate passage times $\hat T(x,y)$ between circuits of low-weight edges encircling each of $x$ and $y$, and shows that $\{T(x,y) - \hat T(x,y) : x,y \in \mathbb{Z}^2\}$ is tight.
\end{remark}

\subsection{Properties of the asymptotic shape}

Below, we characterize the `time constant' $\bar\mu$ in terms of time constants in homogeneous environments, and this gives a representation for the asymptotic shape. Generally, the shape $\overline\Wc$ is the closed convex hull of the two homogeneous shapes (restricted to their respective half-planes) and a symmetric interval on the $\mathbf{e}_2$-axis of width $2\bar{\mu}(\mathbf{e}_2)^{-1}$, where $\ebf_i$ denotes the $i$th coordinate vector. Define
\bea
\begin{aligned}
\Hb_- := \{(x_1,x_2) \in \mathbb{R}^2 : x_1 \le 0\} \quad&\text{and}\quad   \Hb_+ := \{(x_1,x_2) \in \mathbb{R}^2 : x_1 \ge 0\}\ ,\\
\Wc_-:=\{x\in\Hb_-:\mu_-(x)\le1\} \quad&\text{and}\quad \Wc_+:=\{x\in\Hb_+:\mu_+(x)\le1\}\ .
\end{aligned}
\eea

\begin{thm}\label{thm: mu_bar_properties}
Assume that $\E Y_-$ and $\E Y_+$ are finite. The function $\bar \mu:\R^2\to[0,\infty)$ is described by the formula
\begin{equation}\label{eq: mubar_formula}
\bar \mu(x)=\left\{
\begin{aligned}
\min_{a \in \mathbb{R}} \left[ \bar \mu(a \mathbf{e}_2) + \mu_-(x-a\mathbf{e}_2) \right]& \quad\text{ for } x \in \Hb_-\ ,\\
\min_{a \in \mathbb{R}} \left[ \bar \mu(a \mathbf{e}_2) + \mu_+(x-a \mathbf{e}_2) \right] &\quad\text{ for } x \in \Hb_+\ .
\end{aligned}
\right.
\end{equation}
Further, $\bar\mu$ is sub-additive and positive homogeneous, and $\overline\Wc$ equals the closed convex hull
\begin{equation}\label{eq: convex_hull}
\overline\Wc = \textup{convex hull}\Big[\, \Wc_- \cup\, \Wc_+ \cup \{0\}\!\times\!\left[-\bar \mu(\mathbf{e}_2)^{-1}, \bar \mu(\mathbf{e}_2)^{-1}\right]\Big]\ .
\end{equation}
\end{thm}
\noindent
From~\eqref{eq: convex_hull} we also see that if $F_+(0) \geq p_c$ but $F_-(0) < p_c$, then limit shape is a half-plane.

The interval on the $\mathbf{e}_2$-axis described in the last theorem gives the possibility of an additional `pyramid' in the coordinate direction for $\overline\Wc$.
\begin{figure}[htbp]
\begin{center}
\includegraphics[width=0.7\textwidth]{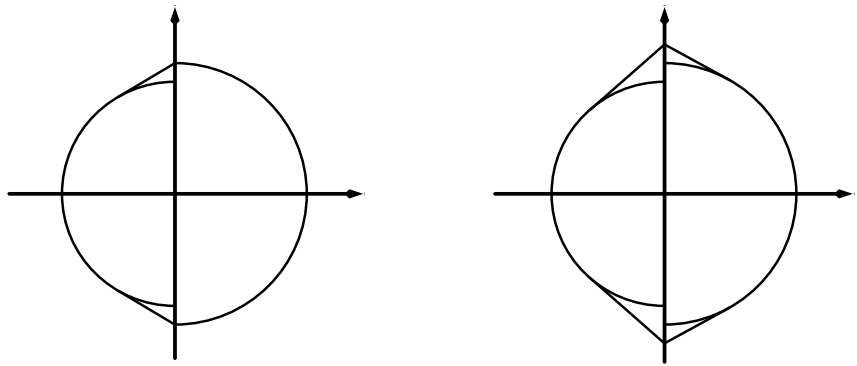}
\end{center}
\caption{
Schematic exhibiting the structure of the asymptotic shape $\overline\Wc$. The left picture is known to be the case when $F_+$ is more variable than $F_-$.
}
\label{fig: shape_thm}
\end{figure}
In this case, optimal paths in the $\mathbf{e}_2$-direction are able to benefit from low-weight edges in both half-planes better than if they were to remain in one of them, and will thus feel an `attraction' towards the $\ebf_2$-axis. However, we will see below that if one of $F_-$ and $F_+$ dominates the other (in a certain concave ordering), then $\bar{\mu}(\ebf_2)$ equals either $\mu_-(\ebf_2)$ or $\mu_+(\ebf_2)$, and the statement in~\eqref{eq: convex_hull} is reduced to
$$
\overline\Wc=\textup{convex hull}\big[\,\Wc_-\cup\,\Wc_+\big]\ ,
$$
with no additional pyramid. We do not completely understand the mechanism that determines whether strict inequality holds. 

For the statement of the next theorem, we say that $F_1$ if \emph{more variable} than $F_2$ if
\[
\int \phi(x)\,\text{d}F_1(x) \leq \int \phi(x)\,\text{d}F_2(x)
\]
for every concave non-decreasing function $\phi:\mathbb{R} \to \mathbb{R}$ for which the two integrals above converge absolutely. In this case we write $F_1 \prec F_2$. This terminology was introduced to first-passage percolation by van den Berg and Kesten~\cite[Definition~2.1]{vdbkes93}, who used it to prove inequalities between time constants for different distributions. Note, in particular, that if $F_-$ stochastically dominates $F_+$ (that is $F_-(x)\le F_+(x)$ for all $x$), then $F_+$ is more variable than $F_-$.

\begin{thm}\label{thm: domination}
Assuming that $\E Y_-$ and $\E Y_+$ are finite,
\[
\bar{\mu}(\ebf_2) \le \min\{\mu_-(\ebf_2),\mu_+(\ebf_2)\}\ ,
\]
where equality holds if either of $F_-$ and $F_+$ is more variable than the other.
\end{thm}

If strict inequality holds above, then $\overline\Wc$ has a pyramid in the coordinate direction $\ebf_2$. The final theorem we state here shows that there are examples that display this behavior.

\begin{thm}\label{thm: strict_inequality}
There exist non-degenerate weight distributions $F_-$ and $F_+$ such that
\[
\bar \mu(\ebf_2) < \min \{\mu_-(\ebf_2),\mu_+(\ebf_2)\}\ .
\]
\end{thm}

In Section~\ref{sec: random_defects} we will prove a related result for the homogeneous model with edge-weights given by $F$ and columnar defects given by $F_0$ introduced at random locations: For a large class of distributions $F$ it suffices that $F_0\prec F$ and $F_0\neq F$ for the time constant in the vertical direction to be strictly smaller than $\mu_F$, regardless of the density at which defects are introduced.

\section{Preliminaries}\label{sec: notation}

To derive properties of the inhomogeneous model, we will when possible rely on known facts about the homogeneous case. Some of these facts will be recalled in this section.

Let $F$ be a distribution function for a probability measure supported on $[0,\infty)$ and let $Y_F$ denote the minimum of $4$ independent random variables distributed as $F$. Let $T_F$ denote travel times on $\Z^2$ in an homogeneous i.i.d.\ environment generated by $F$. A small extension of the arguments in~\cite[Theorem~4]{coxdur81} (also see Sections~2.2 and A of \cite{A13}) show that for every $F$ there is a function $\mu_F:\R^2\to[0,\infty)$ such that for every $x\in\R^2$
\be\label{eq:radialhom}
\lim_{n\to\infty}\frac{T_F (0,nx)}{n}=\mu_F(x)\quad\text{in probability}\ .
\ee
If $\E Y_F<\infty$, then the convergence also holds almost surely and in $L^1$, as a consequence of the Subadditive Ergodic Theorem of~\cite{kingman68}. Kesten~\cite[Theorem~6.1]{kesten86} identified the condition for $\mu_F$ to be non-degenerate: 
\begin{equation}\label{eq: kesten_condition}
\mu_F(x)\neq0 \text{ for some } x \neq 0 \quad\Leftrightarrow\quad\mu_F(x)\neq0\text{ for all }x\neq0\quad\Leftrightarrow\quad F(0)<p_c\ ,
\end{equation}
where $p_c = 1/2$ denotes the critical threshold for bond percolation on the lattice $\Z^2$. $\mu_F$ is a semi-norm on $\R^2$, so for $x=(x_1,x_2)\in\R^2$, $\mu_F(x) \leq |x_1|\mu_F(\mathbf{e}_1) + |x_2|\mu_F(\mathbf{e}_2)$, and thus
\begin{equation}\label{eq: lipschitz}
|\mu_F(x)-\mu_F(y)| \,\leq\, \mu_F(x-y) \,\leq\, \mu_F(\ebf_1)\|x-y\|_1\quad\text{for } x,y \in \mathbb{R}^2\ ,
\end{equation}
where we use $\|\cdot\|_1$ to denote $\ell^1$-distance.

At times we will couple the homogeneous and inhomogeneous models. Given the distribution function $F$, define the right-continuous inverse 
\[
F^{-1}(u):=\min\{x\in\R:F(x)\ge u\}\ .
\] 
If $\xi$ is uniformly distributed on the interval $[0,1]$, then $F^{-1}(\xi)$ is distributed according to $F$. The inhomogeneous collection $\{\tau_e\}_{e\in\Ec}$ of edge-weights of the lattice can thus be obtained by setting $\tau_e = F^{-1}_-(\xi_e)$ for edges with an endpoint in the interior of $\Hb_-$, and $F^{-1}_+(\xi_e)$ for all other edges, where $\{\xi_e\}_{e\in\Ec}$ is a collection of independent random variables uniform on $[0,1]$. This implies
\begin{equation}\label{eq: domination}
\Flow^{-1}(\xi_e) \,\leq\, \tau_e \,\leq\, \Fhi^{-1}(\xi_e)\quad\text{for every }e\in\Ec\ ,
\end{equation}
where $\Flow:=\max\{F_-,F_+\}$ and $\Fhi:=\min\{F_-,F_+\}$. The inequalities for the pointwise coupling in~\eqref{eq: domination} carry over to passage times between sites:
\bea
T_{\Flow}(x,y)\,\le\, T(x,y)\,\le\,T_{\Fhi}(x,y)\quad\text{for all }x,y\in\R^2\ .
\eea
Note, in particular, that if $F_-$ dominates $F_+$ (meaning $F_-(x)\le F_+(x)$ for all $x\ge0$), then $\Fhi=F_-$ and $\Flow=F_+$.

In addition,~\eqref{eq: domination} gives rise to a comparison of means.

\begin{lma}\label{lem: cd_moment_bound}
Let $\Yhi$ denote the minimum of $4$ independent random variables distributed according to $\Fhi$. For any $\alpha>0$
$$
\max\{\E Y_-^\alpha ,\E Y_+^\alpha\}\;\le\;\E \Yhi^\alpha \;\le\;8\,(\E Y_-^\alpha +\E Y_+^\alpha)\ .
$$
In particular, for every $\alpha>0$ and $z\in\R^2$, $\E T(0,z)^\alpha$ is finite if both $\E Y_-^\alpha$ and $\E Y_+^\alpha$ are.
\end{lma}

\begin{proof}
The first inequality holds by $\max\{Y_-,Y_+\} \leq Y_{\textup{dom}}$. The next is from
 \bea
 \begin{aligned}
 \mathbb{P}(Y_\textup{dom} \geq x)\; =\; \Pr\big(F_{\textup{dom}}^{-1}(\xi) >x\big)^4\;&\le\;\big(\Pr(F_-^{-1}(\xi)>x)+\Pr(F_+^{-1}(\xi)>x)\big)^4\\
 &\le\; 8\big(\Pr(F_-^{-1}(\xi)>x)^4+\Pr(F_+^{-1}(\xi)>x)^4\big) \\
 &=\;8 \big( \mathbb{P}(Y_- \geq x ) + \mathbb{P}(Y_+ \geq x) \big)\ .
 \end{aligned}
 \eea
 The proof of the last statement consists of constructing 4 edge-disjoint paths from 0 to $x$ and arguing as in \cite[Lemma~3.1]{coxdur81}.
\end{proof}

The last result we mention here will be used to bound tail probabilities while deriving both radial convergence and the shape theorem (Theorems~\ref{thm: radial_convergence} and~\ref{thm: shape_theorem}). Here and below we let $T_+$ denote the passage time over paths restricted to $\Hb_+$ and not including edges along the $\ebf_2$-axis.

\begin{prop}\label{prop: homogeneousLDE}
For every $\eps>0$ there exist $M=M(\eps)<\infty$ and $\gamma=\gamma(\eps)>0$ such that for every $x\in\R^2$ and $t\ge|x|$
\begin{equation}\label{eq: LDhom}
\Pr\big(T_F(0,x)-\mu_F(x)<-\eps t\big)\,\le\, M e^{-\gamma t}\ .
\end{equation}
Moreover, if $\E Y_+<\infty$, $x\in\Hb_+$ and $q\ge1$, then $M=M(\eps,q)$ can be chosen so that
\begin{equation}\label{eq: LD}
\Pr\big(|T_+(2\mathbf{e}_1,2\mathbf{e}_1+x)-\mu_+(x)|>\eps t\big)\,\le\, M\,\Pr(Y_+>t/M)+\frac{M}{t^q}\ .
\end{equation}
\end{prop}

The former statement in Proposition~\ref{prop: homogeneousLDE} was obtained in~\cite[(1.4)]{grikes84} for $x=n\ebf_1$, and extended to general directions in~\cite[Theorem~3]{A13}. The latter statement is proven in~\cite[Theorem~9]{A13-2} (see~\cite[Theorem~4]{A13} for the corresponding statement for the whole lattice). We remark that the moment condition in the latter part can be replaced by the condition that $\E Y_+^\alpha<\infty$ for some $\alpha>0$, but we will not need to use that.
We further remark that~\eqref{eq: LDhom} and~\eqref{eq: LD} imply that $\sum_{n\ge1}\Pr\big(|T_F(0,nx)-\mu_F(x)|>\eps n\big)<\infty$ for every $\eps>0$, under the condition $\E Y_F<\infty$. This fact has certain relevance for the results of this paper (in particular Theorem~\ref{thm: radial_convergence}) to be obtained under minimal assumptions.

\section{Radial convergence -- Proof of Theorem~\ref{thm: radial_convergence}}

In this section we prove Theorem~\ref{thm: radial_convergence}, and assume throughout that $\E Y_-$ and $\E Y_+$ are finite. By Lemma~\ref{lem: cd_moment_bound}, $\E T(m\mathbf{e}_2,n\mathbf{e}_2) < \infty$ for all $m,n \in \mathbb{Z}$.

\begin{claim}\label{claim: axis}
There exists a constant $\nu\in[0,\infty)$ such that for every $\lambda\in\R$
\begin{equation*}
\bar{\mu}(\lambda \mathbf{e}_2)\,:=\, \lim_{n\to \infty} \frac{T(0,n\lambda \mathbf{e}_2)}{n} \,=\, |\lambda| \nu \quad\text{almost surely and in }L^1\ .
\end{equation*}
\end{claim}

\begin{proof}
Due to translation ergodicity of the environment along the $\ebf_2$-axis, a standard application of the Subadditive Ergodic Theorem shows
\begin{equation*}
 \infty > \nu := \lim_{n \to \infty} \frac{T(0,\pm n\mathbf{e}_2)}{n}  \quad\text{ exists almost surely and in }L^1\ .
\end{equation*}
This proves the claim for $\lambda=\pm1$, for which the limits coincide due to symmetry. More generally, for $\lambda >0$ and $n \in \mathbb{N}$,
\[
\frac{1}{n} T(0,n\lambda \mathbf{e}_2) \,=\, \frac{1}{n} \big( T(0,n\lambda \mathbf{e}_2) - T(0,\lfloor n\lambda \rfloor \mathbf{e}_2) \big) + \frac{\lfloor n\lambda \rfloor}{n} \frac{1}{\lfloor n\lambda \rfloor} \,T(0,\lfloor n\lambda \rfloor \mathbf{e}_2)\ .
\]
The second term converges almost surely and in $L^1$ to $\lambda\nu$. The other term converges to zero almost surely via Borel-Cantelli: for each $\eps>0$,
\begin{equation}\label{eq: mean_trick}
\sum_{n=1}^\infty \mathbb{P}\big(|T(0,n\mathbf{e}_2) - T(0,(n-1)\mathbf{e}_2)| > \eps n\big) \,\leq\, \sum_{n=1}^\infty \mathbb{P}\big(T(0,\mathbf{e}_2) > \eps n\big) \,\leq\, \frac{1}{\eps}\, \E T(0,\mathbf{e}_2)\ ,
\end{equation}
which is finite. This proves almost sure convergence for $\lambda>0$. $L^1$ convergence follows since $\E\big|T(0,n\lambda \mathbf{e}_2) - T(0,\lfloor n\lambda \rfloor \mathbf{e}_2)\big|$ is bounded. The case of $\lambda < 0$ is similar.
\end{proof}

A key observation used in the rest of the proof of Theorem~\ref{thm: radial_convergence} is that
\begin{equation}\label{eq: variation}
T(0,z)=\inf_{k\in\Z}\big[T(0,k\ebf_2)+T_+(k\ebf_2,z)\big]\quad\text{for all }z\in\Hb_+\ .
\end{equation}
(Recall that $T_+$ is the passage time among paths using edges with at least one endpoint in the interior of $\Hb_+$.) To prove this, let $\eps>0$ and choose a path $\Gamma$ from $0$ to $z$ such that $T(\Gamma)$ (the sum of edge-weights for edges in $\Gamma$) is no bigger than $T(0,z) + \eps$. $\Gamma$ has a terminal segment $\Gamma_t$ contained in the open right half-plane (except its initial and possibly its final vertex) from some $k\mathbf{e}_2$ to $z$. Write $\Gamma_i$ for the initial segment of $\Gamma$ up to $k\ebf_2$. Then
\[
T(0,k\mathbf{e}_2) + T_+(k\mathbf{e}_2,z) \;\leq\; T(\Gamma_i) + T_+(\Gamma_t) \;=\; T(\Gamma) \;\leq\; T(0,z) + \eps\ .
\]
Taking infimum over $k$ and sending $\eps\to0$ gives one inequality of \eqref{eq: variation}. For the other, let $k \in \mathbb{Z}$ and choose paths $\Gamma_1$ from 0 to $k\mathbf{e}_2$ and $\Gamma_2$, contained in the interior of $\Hb_+$ except its initial and possibly its final vertex, from $k \mathbf{e}_2$ to $z$ such that $T(\Gamma_1) \leq T(0,k\mathbf{e}_2) + \eps/2$ and $T_+(\Gamma_2) \leq T_+(k\mathbf{e}_2,z) + \eps/2$. The concatenation of $\Gamma_1$ and $\Gamma_2$, written $\Gamma$, is a path from $0$ to $z$, so
\[
T(0,z) \;\leq\; T(\Gamma) \;=\; T(\Gamma_1) + T_+(\Gamma_2) \;\leq\; T(0,k\mathbf{e}_2) + T_+(k\mathbf{e}_2,z) + \eps\ .
\]
This is true for all $k$ and $\eps>0$, so it proves the other inequality.

Returning to the proof of radial convergence, for $x=(x_1,x_2) \in \mathbb{R}^2$ with $x_1>0$ we set
\begin{equation}\label{eq: nu_def}
\nu(x) := \inf_{a \in \mathbb{R}} ~\big[\bar{\mu}(a \mathbf{e}_2) + \mu_+(x-a\mathbf{e}_2)\big]\ ,
\end{equation}
where $\bar{\mu}(a \mathbf{e}_2)$ is defined via Claim~\ref{claim: axis}.

\begin{claim}\label{claim: limsup}
For every $x=(x_1,x_2) \in \mathbb{R}^2$ with $x_1>0$,
\begin{equation*}
\limsup_{n \to \infty} \frac{T(0,nx)}{n} \leq \nu(x) \quad\text{almost surely}\ .
\end{equation*}
\end{claim}

\begin{proof}
Let $a\in \mathbb{R}$ and $\delta>0$. By subadditivity, an upper bound on $T(0,nx)$ is given by
\begin{equation}\label{eq: first_replace}
T(0,an \mathbf{e}_2) + T(an\mathbf{e}_2,an\mathbf{e}_2 + 2\mathbf{e}_1) + T_+(an\mathbf{e}_2 + 2\mathbf{e}_1,nx+2\mathbf{e}_1)+T(nx+2\mathbf{e}_1,nx)\ .
\end{equation}
Just as in \eqref{eq: mean_trick}, both $\frac{1}{n}T(an\mathbf{e}_2,an\mathbf{e}_2+2\mathbf{e}_1)$ and $\frac{1}{n}T(nx+2\mathbf{e}_1,nx)$ vanish almost surely as $n\to\infty$, so almost sure convergence of $\frac{1}{n}T(0,an\mathbf{e}_2)$ to $\bar{\mu}(a\mathbf{e}_2)$ reduces the proof of Claim~\ref{claim: limsup} to showing
\[
\limsup_{n \to \infty} \frac{T_+(an\mathbf{e}_2+2\mathbf{e}_1,nx+2\mathbf{e}_1)}{n} \leq \mu_+(x-a \mathbf{e}_2)\quad\text{almost surely}\ .
\]
To do this, we use the latter part of Proposition~\ref{prop: homogeneousLDE}. Let $y=x-a\ebf_2$, and take $\eps=\delta/|y|$ and let $q=2$ so that the proposition gives a constant $M=M(\eps)$ such that
\[
\mathbb{P}\big(T_+(2\ebf_1,2\ebf_1+ny) > n(\mu_+(y) + \delta)\big) \,\leq\, M\, \mathbb{P}(Y_+ \geq |y|n/M) + \frac{M}{|y|^2n^2}\ .
\]
Since $|y|=|x-a\ebf_2|\ge x_1>0$ and $Y_+$ has finite mean, the sum over $n\in\N$ converges. So, Borel-Cantelli shows that $\frac{1}{n}T_+(an\mathbf{e}_2+2\mathbf{e}_1,nx+2\mathbf{e}_1)>\mu_+(x-a \mathbf{e}_2)+\delta$ for at most finitely many $n$, almost surely and proves Claim~\ref{claim: limsup}.
\end{proof}

Before proving a lower bound matching Claim~\ref{claim: limsup}, we separate a consequence of Proposition~\ref{prop: homogeneousLDE} that we will use.

\begin{claim}\label{claim: deviation}
For every $\eps>0$ and $x=(x_1,x_2)\in\R^2$ with $x_1>0$,
\begin{equation*}
\mathbb{P}\Big(T_+(k\mathbf{e}_2,nx) < \mu_+(nx-k\mathbf{e}_2)-\eps|nx-k\ebf_2| \text{ for infinitely many }(n,k) \in \mathbb{N} \times \mathbb{Z}\Big) = 0\ .
\end{equation*}
\end{claim}

\begin{proof}
Note that $T_+(k\ebf_2,nx)$ equals $T_+(0,nx-k\ebf_2)$ in distribution. Pick $\eps>0$ and let $M=M(\eps)$ and $\gamma=\gamma(\eps)$ be given as in the first part of Proposition~\ref{prop: homogeneousLDE}. Then,
$$
\Pr\big(T_+(0,nx-k\ebf_2) - \mu_+(nx-k\mathbf{e}_2)<-\eps|nx-k\ebf_2|\big)\,\le\,Me^{-\gamma|nx-k\ebf_2|}
$$
for all $n\in\N$ and $k\in\Z$. Since $x_1>0$, as $n$ and $k$ ranges over $\N$ and $\Z$, respectively, $nx-k\ebf_2$ will be in each unit square $z+[0,1)^2$ at most a finite number (say $K$) of times, for each $z\in\Z^2$. Replacing $nx-k\ebf_2$ by a corner of the unit square in which it is in, we find that
$$
\sum_{n\in\N,\,k\in\Z}\Pr\big(T_+(0,nx-k\ebf_2) - \mu_+(nx-k\mathbf{e}_2)<-\eps|nx-k\ebf_2|\big)\,\le\,\sum_{z\in\N\times\Z}2KMe^{-\gamma|z|}\ ,
$$
which is finite. The claim thus follows by Borel-Cantelli.
\end{proof}

\begin{claim}\label{claim: liminf}
For every $x=(x_1,x_2) \in \mathbb{R}^2$ with $x_1>0$,
\bea
\liminf_{n\to\infty}\frac{T(0,nx)}{n}\ge \nu(x)\quad\text{almost surely}\ .
\eea
\end{claim}

\begin{proof}
The proof will proceed through a few different cases. If $\nu(x)=0$ there is nothing to prove, so we may assume that $\nu(x)>0$. This guarantees that $\mu_+\not\equiv0$, because otherwise we could take $a=0$ in the definition of $\nu$; however, $\bar\mu(\ebf_2)$ may still be zero. First assume that $\mu_+ \not\equiv0$ and $\bar \mu(\ebf_2)=0$. Then for every $a \in \mathbb{R}$, symmetry and convexity of $\mu_+$ gives
\begin{equation*}
\mu_+(x-a\mathbf{e}_2) \;=\; \tfrac{1}{2}\big[\mu_+(x_1\mathbf{e}_1 + (x_2-a)\mathbf{e}_2) + \mu_+(x_1\mathbf{e}_1 - (x_2-a)\mathbf{e}_2)\big] \;\geq\; \mu_+(x_1\mathbf{e}_1)\ .
\end{equation*}
In particular $\nu(x)=\mu_+(x_1\ebf_1)>0$ and $\mu_+(nx_1\ebf_1)\le\mu_+(nx-k\ebf_2)$, so by \eqref{eq: variation}, for any $\delta>0$, we have
\begin{align*}
&\mathbb{P}\Big(T(0,nx) < (1-\delta)n\nu(x) \text{ for infinitely many }n\Big) \\
&\quad\leq\; \mathbb{P}\Big(\inf_{k \in \mathbb{Z}} T_+(k\mathbf{e}_2,nx) < (1-\delta)\mu_+(nx_1\ebf_1)\text{ for infinitely many }n\Big) \\
&\quad\leq\; \mathbb{P}\Big(T_+(k\mathbf{e}_2,nx) < (1-\delta)\mu_+(nx-k\mathbf{e}_2) \text{ for infinitely many }(n,k) \in \mathbb{N} \times \mathbb{Z}\Big)\ .
\end{align*}
Since $\mu_+(nx-k\ebf_2)\ge\inf\{\mu_+(y):|y|=1\}|nx-k\ebf_2|$, we may apply Claim~\ref{claim: deviation} (with $\eps=\delta\cdot\inf\{\mu_+(y):|y|=1\}>0$) to find that the above probability equals zero, thus proving Claim~\ref{claim: liminf} when $\mu_+\not\equiv0$ but $\bar{\mu}(\mathbf{e}_2) = 0$.

To complete the proof of Claim~\ref{claim: liminf}, it remains to verify the case $\mu_+ \not\equiv0$ and $\bar\mu(\mathbf{e}_2)>0$. We will use \eqref{eq: variation} and we first bound its right side for large $k$. So, fix $b>0$ such that
\[
b\cdot\bar\mu(\ebf_2) > \nu(x)\ .
\]
By convergence along the $\mathbf{e}_2$-axis, almost surely, for all large $n \in \mathbb{N}$, and $k \in \mathbb{Z}$ for which $|k|>bn$
\[
\frac{T(0,k\mathbf{e}_2)}{n} \;>\; b\, \frac{T(0,k\mathbf{e}_2)}{|k|} \;\geq\; \nu(x)\ ,
\]
giving
\be\label{eq:lower1}
\Pr\left(\inf_{|k|> bn}\frac{T(0,k\ebf_2)}{n}\ge\nu(x)\text{ for all large }n\right)=1\ .
\ee

Next we show that given $\eps>0$, almost surely there exists $N \in \mathbb{N}$ such that for all $n\ge N$
\begin{equation}\label{eq: condition_1}
T(0,k\mathbf{e}_2) \,\geq\, \bar{\mu}(k\mathbf{e}_2) - \eps n/2 \quad\text{for all } k \in \mathbb{Z} \text{ with } |k| \leq bn
\end{equation}
and
\begin{equation}\label{eq: condition_2}
T_+(k\mathbf{e}_2,nx) \,\geq\, \mu_+(nx-k\mathbf{e}_2) - \eps n/2 \quad\text{for all } k \in \mathbb{Z} \text{ with } |k| \leq bn\ .
\end{equation}
Combining~\eqref{eq: condition_1} and~\eqref{eq: condition_2}, we will then have for $n \geq N$,
\[
T(0,k\mathbf{e}_2) + T_+(k\mathbf{e}_2,nx) \,\geq\, \bar{\mu}(k\mathbf{e}_2) + \mu_+(nx-k\mathbf{e}_2) - \eps n \quad\text{for } |k| \leq bn\ .
\]
By the definition of $\nu$ as an infimum,
\[
T(0,k\mathbf{e}_2) + T_+(k\mathbf{e}_2,nx) \,\geq\, \nu(nx) - \eps n \quad\text{for } n \geq N \text{ and } |k| \leq bn\ .
\]
On the other hand, \eqref{eq:lower1} gives $N' \geq N$ such that if $n \geq N'$ then 
\[
T(0,k\mathbf{e}_2) + T_+(k\mathbf{e}_2,nx) \,\geq\, T(0,k\mathbf{e}_2) \,\geq \,\nu(nx) - \eps n \quad\text{for } |k| > bn\ .
\]
By \eqref{eq: variation}, we would have $T(0,nx) \geq \nu(nx) - \eps n$ for $n \geq N'$, completing the proof of Claim~\ref{claim: liminf}.

So, it remains to prove \eqref{eq: condition_1} and \eqref{eq: condition_2}. Fix $\delta$ with
\begin{equation}\label{eq: delta_def}
0\,<\, \delta \,<\, \frac{\eps}{2 (|x|+b)}\ .
\end{equation}
For \eqref{eq: condition_1}, we use convergence along the $\mathbf{e}_2$-axis to, with probability one, find $K$ such that
\[
T(0,k\mathbf{e}_2) \,\geq\, \bar{\mu}(k\mathbf{e}_2) - |k| \delta \quad\text{whenever } |k| \geq K\ .
\]
For $|k| \leq bn$, the choice of $\delta$ ensures that $|k| \delta < \eps n/2$, so $T(0,k\mathbf{e}_2) \geq \bar{\mu}(k\mathbf{e}_2) - \eps n/2$ whenever $K \leq |k| \leq bn$. For $|k| < K$, simply choose $N_1 \geq 2K\bar{\mu}(\mathbf{e}_2)/\eps$. Then if $n \geq N_1$, we have $\bar{\mu}(k\mathbf{e}_2) \leq \eps n /2$, so~\eqref{eq: condition_1} holds for all $n\ge N_1$.

For \eqref{eq: condition_2}, use Claim~\ref{claim: deviation} to find $N_2 \geq N_1$ such that if $n \geq N_2$ then
\[
T_+(k\mathbf{e}_2,nx) \,\geq\, \mu_+(nx-k\mathbf{e}_2)-\delta|nx-k\ebf_2| \quad\text{for all } k \in \mathbb{Z}\ .
\]
Again, for $|k| \leq bn$, the choice of $\delta$ in \eqref{eq: delta_def} gives 
\[
\delta|nx-k\mathbf{e}_2| \;\leq\; \delta n(|x| + |k|/n) \;\leq\; \delta n(|x|+b) \;<\; \eps n/2\ ,
\]
and so~\eqref{eq: condition_2} holds for $n \geq N_2$, and the proof of Claim~\ref{claim: liminf} is complete.
\end{proof}

We have now shown almost sure convergence of $\frac{1}{n}T(0,nx)$ when $x_1 \geq 0$. The case $x_1<0$ is similar, and the only modifications necessary are to replace $T_+$ with $T_-$, the passage time for paths using only edges with at least one endpoint in the interior of $\Hb_-$, as well as equation \eqref{eq: variation} with its obvious analogue and $\mu_+$ with $\mu_-$.

To complete the proof of Theorem~\ref{thm: radial_convergence}, we must prove $L^1$-convergence. (This was already remarked when $x_1=0$ in Claim~\ref{claim: axis}.) We use dominated convergence, bounding $T(0,nx)$ above by $T_{\text{dom}}(0,nx)$, where $T_{\text{dom}}$ is  the passage time in the homogeneous environment with edge-weights distributed as $F_{\text{dom}}$. By Lemma~\ref{lem: cd_moment_bound}, the assumption $\max\{\E Y_-,\E Y_+\} < \infty$ implies that $\E\Yhi < \infty$, where $\Yhi$ is as in Lemma~\ref{lem: cd_moment_bound}. Therefore $L^1$-convergence in homogeneous environments (mentioned below \eqref{eq:radialhom}) completes the proof of Theorem~\ref{thm: radial_convergence}.

\begin{remark}\label{rem: three_f}
The reader may verify that the above proof goes through, essentially word for word, if edges in the left and right half-planes are assigned weights according to $F_-$ and $F_+$, respectively, but edges on the vertical axis are assigned weights according to a distribution $F_0$. In this case, one may require, for example, $\E Y_-, \E Y_+$ and $\E Y_0$ to be finite.
\end{remark}

\section{The time constant and the asymptotic shape}\label{sec: time_constant}

We aim in this section to prove Theorem~\ref{thm: mu_bar_properties}. Assume then that $\E Y_-$ and $\E Y_+$ are finite, so that the limit in Theorem~\ref{thm: radial_convergence} exists. We begin with a simple observation.

\begin{lma}\label{lem: mubar_bound}
The time constant satisfies
\begin{equation*}
\bar\mu(x)\le\left\{
\begin{aligned}
&\mu_-(x)&\text{for }x\in\Hb_-\ ,\\
&\mu_+(x)&\text{for }x\in\Hb_+\ .
\end{aligned}
\right.
\end{equation*}
In particular, $\bar\mu(\ebf_2)\le\min\{\mu_-(\ebf_2),\mu_+(\ebf_2)\}$.
\end{lma}

\begin{proof}
First note that for $x\in\Hb_+$ the subadditive property gives that
\be\label{eq: mubar_bound}
T(0,nx) \,\leq\, T(0,2\mathbf{e}_1) + T_+(2\mathbf{e}_1, 2\mathbf{e}_1 + nx) + T(2\mathbf{e}_1+ nx,nx)\ .
\ee
We now apply Proposition~\ref{prop: homogeneousLDE} to find that $\Pr\big(|T_+(2\mathbf{e}_1, 2\mathbf{e}_1 + nx)-\mu_+(nx)|>\eps n\big)\to0$ as $n\to\infty$.
Together with~\eqref{eq: mean_trick}, dividing by $n$ and sending $n\to\infty$ in~\eqref{eq: mubar_bound}, the right side converges in probability to $\mu_+(x)$. Since the left side, after division by $n$, converges to $\bar\mu(x)$, we find $\bar\mu(x)\le\mu_+(x)$. A similar argument shows $\bar\mu(x)\le\mu_-(x)$ when $x\in\Hb_-$, and the statement follows.
\end{proof}

We now begin the proof of Theorem~\ref{thm: mu_bar_properties}, starting with formula~\eqref{eq: mubar_formula}.

\begin{proof}[{\bf Proof of~(\ref{eq: mubar_formula})}]
By Claims~\ref{claim: limsup} and~\ref{claim: liminf} in the proof of Theorem~\ref{thm: radial_convergence}, formula \eqref{eq: mubar_formula} holds (with minimum replaced by infimum) when $x$ is not a multiple of $\mathbf{e}_2$. If $x=\lambda\ebf_2$ for some $\lambda\in\R$, then
\[
\bar{\mu}(\lambda \mathbf{e}_2) \;=\; \bar{\mu}(\lambda \mathbf{e}_2) + \mu_{\pm}(\lambda \mathbf{e}_2 - \lambda \mathbf{e}_2) \;\geq\; \inf_{a \in \mathbb{R}} ~[\bar{\mu}(a \mathbf{e}_2) + \mu_{\pm}(\lambda \mathbf{e}_2 - a\mathbf{e}_2)]\ .
\]
On the other hand, by Lemma~\ref{lem: mubar_bound} and the scaling $\bar\mu(\lambda\ebf_2)=|\lambda|\bar\mu(\ebf_2)$,
\[
\bar{\mu}(a\mathbf{e}_2) + \mu_{\pm}(\lambda \mathbf{e}_2-a\mathbf{e}_2) \;\geq\; |a| \bar{\mu}(\mathbf{e}_2) + |\lambda - a| \bar{\mu}(\mathbf{e}_2) \;\geq\; \bar{\mu}(\lambda \mathbf{e}_2) \quad\text{for all } a \in \mathbb{R}\ .
\]
We conclude that $\bar\mu(x)=\inf_{a\in\R}\big[\bar\mu(a\ebf_2)+\mu_\pm(x-a\ebf_2)\big]$ for all $x\in\Hb_\pm$.

Last we must show that the infimum is actually attained. So without loss in generality, let $x \in \Hb_+$ and consider the  continuous function $a \mapsto \bar{\mu}(a\mathbf{e}_2) + \mu_+(x-a\mathbf{e}_2)$. If $\bar{\mu}(\mathbf{e}_2) \neq 0$ or both $\mu_+ \not\equiv0$ and $x=(x_1,x_2)$ has $x_1>0$ then this function approaches $\infty$ as $|a|\to \infty$ and so has a minimum. Otherwise, the minimum (zero) is attained at any $a$.
\end{proof}

We next show that the function $\bar{\mu}$ retains some properties of a semi-norm. In particular, semi-norms are convex, which thus is the case for both $\mu_-$ and $\mu_+$.

\begin{prop}\label{prop: mubar_properties}
The time constant satisfies the following properties.
\begin{enumerate}[\quad a)]
\item $\bar \mu(\lambda x) = \lambda\,\bar \mu(x)$ for $\lambda\ge0$ and $x \in \mathbb{R}^2$.
\item $\bar\mu(x+y)\le\bar\mu(x)+\bar\mu(y)$ for $x,y\in\R^2$.
\item $|\bar \mu(x) - \bar{\mu}(y)| \le \max\{\mu_-(\ebf_2),\mu_+(\ebf_2)\}\|x-y\|_1$ for $x,y\in\R^2$.
\item $\bar \mu(\mathbf{e}_2) \neq 0\quad\Leftrightarrow\quad\bar\mu(x)\neq0\text{ for every }x\neq0\quad\Leftrightarrow\quad\max\{F_-(0),F_+(0)\} < p_c$.
\end{enumerate}
Moreover, for $x=\ebf_2$, part {a)} extends to $\bar\mu(\lambda\ebf_2)=|\lambda|\bar\mu(\ebf_2)$ for all $\lambda\in\R$.
\end{prop}

\begin{proof}
The first three properties follow from formula~\eqref{eq: mubar_formula}. First, for $\lambda\in\R$, Claim~\ref{claim: axis} shows that $\bar{\mu}(\lambda\mathbf{e}_2) = |\lambda|\bar{\mu}(\mathbf{e}_2)$. Next, for any $\lambda\ge0$ and $x \in \Hb_+$ we have from~\eqref{eq: mubar_formula} that
\bea
\begin{aligned}
\bar{\mu}(\lambda x) \;&=\; \min_{a \in \mathbb{R}} \big[\bar{\mu}(a\ebf_2) + \mu_+(\lambda x - a\ebf_2)\big] \\
&=\; \lambda \min_{a \in \mathbb{R}} \big[\bar{\mu}((a/\lambda)\ebf_2) + \mu_+(x-(a/\lambda)\ebf_2)\big] \;=\;\lambda\,\bar{\mu}(x)\ .
\end{aligned}
\eea
The case $x\in\Hb_-$ is analogous, so this proves part~\emph{a)}.

For part~\emph{b)}, assume first that $x,y\in\Hb_+$. For $a,b\in\R$ we have by~\eqref{eq: mubar_formula} that
\[
\bar\mu(x+y) \,\le\, \bar\mu((a+b)\ebf_2)+\mu_+(x+y-(a+b)\ebf_2)\ .
\]
Claim~\ref{claim: axis} gives $\bar{\mu}((a+b)\ebf_2) = |a+b|\bar\mu(\ebf_2) \le \bar{\mu}(a\mathbf{e}_2) + \bar{\mu}(b\mathbf{e}_2)$, and subadditivity of $\mu_+$ implies 
\[
\bar\mu(x+y) \,\le\, \bar\mu(a\ebf_2)+\mu_+(x-a\ebf_2)+\bar\mu(b\ebf_2)+\mu_+(y-b\ebf_2)\ .
\]
Taking minimum of both $a$ and $b$ over $\R$ proves~\emph{b)} for $x,y \in \Hb_+$; the case $x,y\in\Hb_-$ is analogous. For the general case, let $x=(x_1,x_2)\in\Hb_-$ and $y=(y_1,y_2)\in\Hb_+$, and assume that $x+y\in\Hb_+$ (the case $x+y\in\Hb_-$ is again analogous). Writing $x+y=(0,x_2)+(x_1+y_1,y_2)$ we find that
$$
\bar\mu(x+y)\,\le\,\bar\mu(0,x_2)+\bar\mu(x_1+y_1,y_2)\ ,
$$
so it suffices to show that $\bar\mu(0,x_2)\le\bar\mu(x)$ and $\bar\mu(x_1+y_1,y_2)\le\bar\mu(y)$. By convexity of $\mu_\pm$, for $(z_1,z_2) \in \mathbb{R}^2$, $\mu_\pm(0,z_2)\le\frac{1}{2}\big[\mu_\pm(z_1,z_2)+\mu_\pm(-z_1,z_2)\big]=\mu_\pm(z_1,z_2)$. So by~\eqref{eq: mubar_formula}, for some $a\in\R$,
\bea
\begin{aligned}
\bar\mu(x)\;&=\;\bar\mu(a\ebf_2)+\mu_-(x-a\ebf_2)\;\ge\;\bar\mu(0,a)+\mu_-(0,x_2-a)\\
&\ge\;\bar\mu(0,a)+\bar\mu(0,x_2-a)\;\ge\;\bar\mu(0,x_2)\ ,
\end{aligned}
\eea
where the second inequality uses Lemma~\ref{lem: mubar_bound}. Using convexity of $\mu_+$, for $t\in[0,1]$,
\bea
\begin{aligned}
\mu_+(ty_1,y_2)\;&=\;\mu_+\big(t(y_1,y_2)+(1-t)(0,y_2)\big)\\
&\le\; t\mu_+(y_1,y_2)+(1-t)\mu_+(0,y_2)\;\le\;\mu_+(y_1,y_2)\ .
\end{aligned}
\eea
Via~\eqref{eq: mubar_formula}, since $0\le x_1+y_1\le y_1$ by assumption, we now conclude that for every $a\in\R$
$$
\bar\mu(x_1+y_1,y_2)\;\le\;\bar\mu(a\ebf_2)+\mu_+(x_1+y_1,y_2-a)\;\le\;\bar\mu(a\ebf_2)+\mu_+(y-a\ebf_2)\ .
$$
Minimizing over $a\in\R$ gives us $\bar\mu(x_1+y_1,y_2)\le\bar\mu(y)$, completing the proof of part~\emph{b)}.

Part~\emph{c)} is obtained from repeated use of part~\emph{b)}: for $x,y\in\R^2$,
$$
|\bar\mu(x)-\bar\mu(y)|\;\le\;\bar\mu(x-y)\;\le\;\bar\mu((x_1-y_1)\ebf_1)+\bar\mu((x_2-y_2)\ebf_2)\ ,
$$
which by part~\emph{a)} and Lemma~\ref{lem: mubar_bound} is at most $\max\{\mu_-(\ebf_2),\mu_+(\ebf_2)\}\|x-y\|_1$.

It remains to prove part~\emph{d)}. If $\bar\mu(x)\neq0$ for every $x\neq0$, then so is the case for $x=\ebf_2$. We therefore show that $\bar\mu(\ebf_2)\neq0$ implies $\max\{F_-(0),F_+(0)\}<p_c$, and that this in turn implies $\bar\mu(x)\neq0$ for every $x\neq0$. If $\bar\mu(\ebf_2)\neq0$ then $\mu_-(\ebf_2)$ and $\mu_+(\ebf_2)$ are nonzero; otherwise, this would contradict Lemma~\ref{lem: mubar_bound}. By~\eqref{eq: kesten_condition}, $F_-(0)<p_c$ and $F_+(0)<p_c$, which proves the first implication. If we instead assume that $\max\{F_-(0),F_+(0)\}<p_c$, \eqref{eq: kesten_condition} implies that the time constant $\mulow$ with respect to a homogeneous environment with edge-weights distributed as $\Flow$ (defined below~\eqref{eq: domination}) is nonzero for all $x\neq0$. However, by~\eqref{eq: domination}, $\bar\mu(x)\ge\mulow(x)$, which completes the proof of part~\emph{d)}.
\end{proof}

Based on the characterization of $\bar\mu$ and its subsequent properties, we next prove some properties of the asymptotic shape
$\overline\Wc=\{x\in\R^2: \bar{\mu}(x)\le1\}$, and end the proof of Theorem~\ref{thm: mu_bar_properties} by verifying the formula~\eqref{eq: convex_hull}.

\begin{prop}\label{prop: shape_properties}
The asymptotic shape satisfies the following properties.
\begin{enumerate}[\quad a)]
\item $\overline\Wc$ is a closed convex set with non-empty interior.
\item $\overline\Wc$ is compact if and only if $\max\{F_-(0),F_+(0)\}<p_c$.
\item $\overline\Wc$ equals the closed convex hull of the union of $\Wc_-$ and $\Wc_+$, and the straight line segment $\{0\}\times\big[-\bar{\mu}(\ebf_2)^{-1},\bar{\mu}(\ebf_2)^{-1}\big]$.
\end{enumerate}
\end{prop}

\begin{proof}
We first show that $\overline\Wc$ is a closed convex set with non-empty interior. These are all easy consequences of the derived properties of $\bar\mu$. By part~\emph{c)} of Proposition~\ref{prop: mubar_properties}, $\bar{\mu}$ is a continuous function on $\mathbb{R}^2$, and so $\overline\Wc = \{x\in\R^2 : \bar{\mu}(x) \leq 1\}$ is closed. By parts~\emph{a)} and~\emph{b)} of Proposition~\ref{prop: mubar_properties},
$$
\bar\mu(tx+(1-t)y)\,\le\, t\bar\mu(x)+(1-t)\bar\mu(y) \quad\text{for all } x,y \in \mathbb{R}^2 \text{ and } t \in [0,1]\ .
$$
Thus if $x,y\in\overline\Wc$, then also $tx+(1-t)y\in\overline\Wc$, proving convexity. Finally, by Lemma~\ref{lem: mubar_bound} and~\eqref{eq: lipschitz},
$$
\bar\mu(x)\;\le\;\max\{\mu_-(x),\mu_+(x)\}\;\le\;\max\{\mu_-(\ebf_2),\mu_+(\ebf_2)\}\|x\|_1\ ,
$$
showing that $\overline\Wc$ contains $\{x\in\R^2:\|x\|_1\le r\}$ for small values of $r>0$. This means that $\overline\Wc$ has nonempty interior.

We continue with part~\emph{b)}, which we will derive from part~\emph{d)} of Proposition~\ref{prop: mubar_properties}. If either $F_-(0)\ge p_c$ or $F_+(0)\ge p_c$, then $\bar\mu(\ebf_2)=0$ and therefore $\lambda\ebf_2\in\overline\Wc$ for every $\lambda\in\R$, so $\overline\Wc$ is unbounded. Suppose conversely that $\max\{F_-(0),F_+(0)\}<p_c$. Since $\bar\mu$ is continuous, it attains its infimum on the unit circle $\{x\in\R^2:|x|=1\}$, which has to be positive; otherwise would contradict part~\emph{d)} of Proposition~\ref{prop: mubar_properties}. Consequently, for $|x| > \sup_{|y|=1}[\bar{\mu}(y)]^{-1}$, we have $\bar\mu(x)=|x|\bar\mu(x/|x|)>1$ and $\overline\Wc$ is bounded, hence compact.

It remains to prove part~\emph{c)} and we first show that $\overline\Wc$ contains $\Wc_-$, $\Wc_+$ and $\{0\}\!\times\!\big[-\bar{\mu}(\ebf_2)^{-1},\bar{\mu}(\ebf_2)^{-1}\big]$.
Containment of $\Wc_-$ and $\Wc_+$ in $\overline\Wc$ follows from Lemma~\ref{lem: mubar_bound} in that $\bar\mu(x)\le\mu_\pm(x)\le1$ for every $x\in\Wc_\pm$, so $x\in\overline\Wc$. Containment of the interval follows from scaling: $\bar\mu(\lambda\ebf_2)=|\lambda|\bar\mu(\ebf_2)\le1$ for each $\lambda\in\big[-\bar{\mu}(\ebf_2)^{-1},\bar{\mu}(\ebf_2)^{-1}\big]$.

Conversely, we show that $\overline\Wc$ is a subset of the closed convex hull. Without loss in generality, let $x\in \overline\Wc \cap \Hb_+$ and let $a_x$ be such that $\bar{\mu}(x) = \bar{\mu}(a_x\mathbf{e}_2) + \mu_+(x-a_x\mathbf{e}_2)$. If $\mu_+\equiv0$, then $\mu_+(x)=0$, so $x\in\Wc_+$. Assume instead that both $\mu_+\not\equiv0$ and $\bar\mu(\ebf_2)\neq0$. To begin, if $a_x=0$ then $\mu_+(x)=\bar{\mu}(x)\leq 1$, so $x \in\Wc_+$. If instead $x=a_x\mathbf{e}_2$, then $|x|=|a_x| \leq \bar{\mu}(\mathbf{e}_2)^{-1}$ and $x \in \{0\}\!\times\! [-\bar{\mu}(\mathbf{e}_2)^{-1},\bar{\mu}(\mathbf{e}_2)^{-1}]$. If neither $a_x=0$ nor $x-a_x\ebf_2=0$, then $\bar\mu(a_x\ebf_2)$ and $\mu_+(x-a_x\ebf_2)$ are both nonzero, and
\[
x = \frac{\bar{\mu}(a_x\mathbf{e}_2)}{\bar{\mu}(x)} \left[ \frac{\bar{\mu}(x)}{\bar{\mu}(a_x\mathbf{e}_2)}a_x\mathbf{e}_2\right]+ \frac{\mu_+(x-a_x\mathbf{e}_2)}{\bar{\mu}(x)} \left[ \frac{\bar{\mu}(x)}{\mu_+(x-a_x\mathbf{e}_2)}(x-a_x\mathbf{e}_2)\right]\ ,
\]
which is a convex combination of a point in $\{0\}\!\times\! [-\bar{\mu}(\mathbf{e}_2)^{-1},\bar{\mu}(\mathbf{e}_2)^{-1}]$ (the first bracketed term) and one in $\Wc_+$ (the second). 

The last possibility is $\mu_+\not\equiv0$ but $\bar{\mu}(\mathbf{e}_2)=0$. In this case we show that $x$ is a limit of convex combinations: For $n \in \mathbb{N}$, let $x_n = \frac{1}{n}(na_x\mathbf{e}_2) + (1-\frac{1}{n})(x-a_x\mathbf{e}_2)$. Since $\mu_+(x-a_x\mathbf{e}_2) = \bar{\mu}(x) \leq 1$, this is a convex combination of an element of $\{0\}\!\times\! [-\bar{\mu}(\mathbf{e}_2)^{-1},\bar{\mu}(\mathbf{e}_2)^{-1}]$ and an element of $\Wc_+$. As $n \to \infty$, $x_n \to x$, showing that $x$ lies in the closed convex hull of these sets.
\end{proof}

The proof of Theorem~\ref{thm: mu_bar_properties} is now complete by virtue of~\eqref{eq: mubar_formula} and Proposition~\ref{prop: shape_properties}.

\section{Simultaneous convergence -- Proof of Theorem~\ref{thm: shape_theorem}}\label{sec: shape_theorem}

We proceed at this point with a proof of Theorem~\ref{thm: shape_theorem} -- the shape theorem for inhomogeneous environments. It will follow from now standard arguments. We assume throughout that $\E Y_-^2$ and $\E Y_+^2$ are finite, and first prove~\eqref{eq: inverted_shape}. If either $\E Y_-^2$ or $\E Y^2_+$ is infinite, then~\eqref{eq: inverted_shape} cannot hold since $T(0,z)$ is bounded below by the minimum of the four weights of edges adjacent to $z$, each of which is distributed as $F_+$ when $z \in \mathbb{Z}_+\times \mathbb{Z}$. Then we apply Borel-Cantelli with
$$
\sum_{z\in\Z^2}\Pr\big(T(0,z)>\bar\mu(z)+\eps|z|\big)\,\ge\,\sum_{z\in\Z_\pm\!\times\Z}\Pr(Y_\pm>\eps|z|)\,=\,\infty\ .
$$


\begin{proof}[{\bf Proof of~(\ref{eq: inverted_shape})}]
We first need to show that $T(0,z)$ may not be `too large' for more than a finite number of points in $\Z^2$. To quantify `large', we will fix a constant $M$ already at this stage: Recall that $T_\textup{dom}$ denotes passage times in a homogeneous environment distributed as $\Fhi$. By the latter part of Proposition~\ref{prop: homogeneousLDE} (for $\eps=1$ and $q=3$) there is a finite constant $M$, only depending on $\Fhi$, such that for every $x\in\R^2$ and $t\ge|x|$
\begin{equation}\label{eq: LD_dom}
\Pr\big(T_\textup{dom}(0,x)>Mt\big)\,\le\, M\,\Pr(\Yhi>t/M)+Mt^{-3}\ .
\end{equation}
(We have here used that passage times in a half-plane dominate those in the whole plane.)

To argue for~\eqref{eq: inverted_shape}, let $\eps>0$ and fix $\delta$ such that
 \[
0\, <\, \delta\, <\, \frac{\eps}{6} \min\left\{ \mu_-(\mathbf{e}_1)^{-1}, \mu_+(\mathbf{e}_1)^{-1},2, (2M)^{-1}\right\}\ .
 \]
Choose a finite set of unit vectors $u_1,\ldots,u_N$ such that every $u\in \mathbb{R}^2$ with $|u|=1$ satisfies $|u-u_i|<\delta$ for some $i$. For each $z\in\Z^2$ and $i=1,2,\ldots,N$,
\begin{equation}\label{eq: breakdown_1}
|T(0,z)-\bar\mu(z)| \,\le\, |T(0,|z|u_i)-\bar\mu(|z|u_i)|+|\bar\mu(|z|u_i)-\bar\mu(z)|+T(|z|u_i,z)\ .
\end{equation}
According to Theorem~\ref{thm: radial_convergence} there is an almost surely finite constant $K_1$ such that for all $i=1,\ldots,N$ and $|z|\ge K_1$ we have $|T(0,|z|u_i)-|z|\bar\mu(u_i)|<\delta|z|<\eps|z|/3$. For every $z\in\Z^2$ there is at least one $i$ for which $|z-|z|u_i|\le\delta|z|$, and therefore
\[
|\bar\mu(|z|u_i)-\bar\mu(z)| \,\leq\, \| |z|u_i-z\|_1 \max\{\mu_-(\ebf_1),\mu_+(\ebf_1)\}\ ,
\]
due to part \emph{c)} of Proposition~\ref{prop: mubar_properties}. Comparing $\ell^1$ and $\ell^2$ norms leaves
\[
|\bar\mu(|z|u_i)-\bar\mu(z)|\;\leq\; 2\delta |z| \max\{\mu_-(\mathbf{e}_1),\mu_+(\mathbf{e}_1)\} \;\leq\; \eps|z|/3\ .
\]
Last, if we only show that
\be\label{eq: CDsum}
\sum_{z\in\Z^2}\Pr\big(T(|z|u_i,z)>\eps|z|/3\big)<\infty\ ,
\ee
for $i$ chosen so that $|z-|z|u_i| < \delta|z|$, Borel-Cantelli would give a random, but almost surely finite, constant $K_2$ such that if $|z| \geq K_2$ then $T(|z|u_i,z) \leq \eps|z|/3$. Along with the other bounds through~\eqref{eq: breakdown_1}, $|T(0,z)-\bar{\mu}(z)| \leq \eps |z|$ for $|z| \geq \max\{K_1,K_2\}$ and complete be the proof of~\eqref{eq: inverted_shape}.

To verify~\eqref{eq: CDsum}, recall that $T(x,y)\le T_\textup{dom}(x,y)$ for every $x,y\in\R^2$ by~\eqref{eq: domination}. Since the choice of $\delta$ ensures that $\eps|z|/3>M\delta|z|$ and $\delta|z|\geq|z-|z|u_i|$, using~\eqref{eq: LD_dom} we arrive at
\bea
\begin{aligned}
\sum_{z\in\Z^2}\Pr\big(T(|z|u_i,z)>\eps|z|/3\big)\;&\le\;\sum_{z\in\Z^2}\Pr\big(T_\textup{dom}(|z|u_i,z)>M\delta|z|\big)\\
&\le\;\sum_{z\in\Z^2}\Big[M\,\Pr(\Yhi>\delta|z|/M)+M(\delta|z|)^{-3}\Big]\\
&\le\;\sum_{n\in\N}\sum_{\|z\|_1=n}\Big[M\,\Pr(\Yhi>\delta n/(2M))+M(\delta n/2)^{-3}\Big]\ ,
\end{aligned}
\eea
where we also use that $\|z\|_1\le2|z|$. Since the number of points in $\Z^2$ with $\ell^1$-norm $n$ is $4n$, the above sum is finite when $\E\Yhi^2$ is. By Lemma~\ref{lem: cd_moment_bound}, the latter coincides with finiteness of $\E Y_-^2$ and $\E Y_+^2$. As this was assumed,~\eqref{eq: CDsum} follows and therewith~\eqref{eq: inverted_shape}.
\end{proof}

We proceed with the second statement of Theorem~\ref{thm: shape_theorem}, and assume for that, in addition, $\max\{F_-(0),F_+(0)\}<p_c$. That $\overline\Wc$ in this case is convex, compact and has non-empty interior was seen in Proposition~\ref{prop: shape_properties}. It remains to prove the concluding inclusion formula of Theorem~\ref{thm: shape_theorem}, which will follow from~\eqref{eq: inverted_shape} via a straightforward inversion argument. In a first step, we prove a discretized inclusion formula.

\begin{claim*}
With probability 1, for every $\eps>0$ and sufficiently large $t$
\be\label{eq: discrete_inclusion}
\{z\in\Z^2:\bar\mu(z)\le(1-\eps)t\}\subset\{z\in\Z^2:T(0,z)\le t\}\subset\{z\in\Z^2:\bar\mu(z)\le(1+\eps)t\}\ .
\ee
\end{claim*}

\begin{proof}
Let $\eps>0$, and introduce the set $\mathcal{Z}_\eps:=\big\{z\in\Z^2:\tfrac{1}{1+\eps}\,\bar\mu(z)\le T(0,z)\le \tfrac{1}{1-\eps}\,\bar\mu(z)\big\}$. Given $z\in \mathcal{Z}_\eps$, note that $\bar\mu(z)\le(1-\eps)t$ implies that $T(0,z)\le t$, which in turn would imply that $\bar\mu(z)\le(1+\eps)t$. In other words, the restriction of~\eqref{eq: discrete_inclusion} to $\mathcal{Z}_\eps$ holds for all $t\ge0$, and only points not contained in $\mathcal{Z}_\eps$ may cause~\eqref{eq: discrete_inclusion} to fail. So we will use~\eqref{eq: inverted_shape} to show that $\Z^2\setminus \mathcal{Z}_\eps$ is almost surely finite.

Under the assumption $\max\{F_-(0),F_+(0)\}<p_c$, the function $\bar\mu$ is bounded away from both zero and infinity on the unit circle $\{x\in\R^2:|x|=1\}$. If follows that $\bar\mu$ is equivalent to the Euclidean metric on $\R^2$. Using~\eqref{eq: inverted_shape} we find an almost surely finite constant $K$ such that
$$
\big(1-\tfrac{\eps}{1+\eps}\big)\bar\mu(z)\;\le\; T(0,z)\;\le\;\big(1+\tfrac{\eps}{1-\eps}\big)\bar\mu(z) \quad\text{for all }|z|\ge K\ ,
$$
which is equivalent to inclusion in $\mathcal{Z}_\eps$.
\end{proof}

To go from~\eqref{eq: discrete_inclusion} to the inclusion formula of Theorem~\ref{thm: shape_theorem}, it suffices to appeal to Lipschitz continuity of $\bar\mu$: If $z\in\Z^2$ and $x\in[0,1)^2$, then $|\bar\mu(z+x)-\bar\mu(z)|$ is bounded by $2\max\{\mu_-(\ebf_2),\mu_+(\ebf_2)\}$. So, if~\eqref{eq: discrete_inclusion} holds for some $\eps>0$ and all large enough $t$, then
$$
(1-2\eps)\overline\Wc\subset\tfrac{1}{t}\Wc_t\subset(1+2\eps)\overline\Wc\quad\text{for all large enough }t\ .
$$
This ends the proof of Theorem~\ref{thm: shape_theorem}.

\section{Comparisons between time constants in the vertical direction}

We first aim to prove Theorem~\ref{thm: domination}, and so we recall some notation. As in Section~\ref{sec: notation}, let $\{\xi_e\}_{e \in \mathcal{E}}$ be a collection of random variables independent and uniform on $[0,1]$, and let $F^{-1}(x) = \min \{ y \in \mathbb{R} : F(y) \geq x\}$. Set $\tau_e^{\pm} = F_{\pm}^{-1}(\xi_e)$ for each $e\in\mathcal{E}$, and let $\tau_e$ equal $\tau_e^-$ or $\tau_e^+$ depending on whether $e$ has at least one endpoint in the interior of $\Hb_-$ or not. This construction produces a coupling between three environments, two homogeneous in which each edge-weight is distributed as $F_\pm$, and one inhomogeneous.

\begin{proof}[\bf Proof of Theorem~\ref{thm: domination}]
Assume that $\E Y_-$ and $\E Y_+$ are both finite. The upper bound $\bar\mu(\ebf_2)\le\min\{\mu_-(\ebf_2),\mu_+(\ebf_2)\}$ was already proved in Lemma~\ref{lem: mubar_bound}. So we will assume $F_+ \prec F_-$ and prove $\bar\mu(\ebf_2)\ge\min\{\mu_-(\ebf_2),\mu_+(\ebf_2)\}$, arguing similarly to~\cite[Theorem~2.9(a)]{vdbkes93}.

Given an integer $N\geq 1$, let $E_N$ denote the set of edges with both endpoints in $[-N,N]^2$. For a set of edge weights $\sigma = \{\sigma_e\}_{e \in E_N}$ and path $\Gamma \subset E_N$, let $T_\sigma(\Gamma) = \sum_{e \in \Gamma}\sigma_e$ and define $T_\sigma^N(x,y)$ as the minimum of $T_\sigma(\Gamma)$ over all such $\Gamma$ connecting $x$ and $y$. Fix an enumeration $e_1, e_2, \ldots, e_{|E_N|}$ of $E_N$ which goes through the edges of the right half-plane first. For each $j=0, \ldots, |E_N|$, let $\sigma^j=\{\sigma_e^j\}_{e\in E_N}$ be the family given by $\sigma_{e_i}^j = \tau_{e_i}^+$ for $i \leq j$ and $\tau_{e_i}^-$ otherwise. Last, let $\sigma^{j,t}=\{\sigma_e^{j,t}\}_{e\in E_N}$ equal $\{\sigma_e^j\}$ except $\sigma_{e_j}^{j,t} = t$, and let $g_j(t) = T_{\sigma^{j,t}}^N(x,y)$. As a function of $t$, $g_j$ is the minimum of increasing linear functions, so it is non-decreasing and concave.

For $N \geq \|x\|_1+\|y\|_1+1$, $\E T_{\sigma^j}^N(x,y)$ is finite from Lemma~\ref{lem: cd_moment_bound}. Therefore, by Fubini's theorem, for almost every realization of $\{(\tau_e^+,\tau_e^-)\}_{e\in E_N}$, the integrals $\int g_j(t) \,\text{d}F_+(t)$ and $\int g_j(t) \,\text{d}F_-(t)$ are finite. So using the definition of more variable,
\[
\E T_{\sigma^{j-1}}^N(x,y) \;=\; \E g_j(\tau_{e_j}^-) \;\geq\; \E g_j(\tau_{e_j}^+) \;=\; \E T_{\sigma^j}^N(x,y)\quad\text{for } j = 1, \ldots, |E_N|\ .
\]
In particular,
\[
\E T_{F_-}^N(x,y) \;\geq\; \E T^N(x,y) \;\geq\; \E T_{F_+}^N(x,y)\ ,
\]
where $T^N$ and $T_F^N$ are the restrictions of $T$ and $T_F$ to $[-N,N]^2$ as above. By monotone convergence, $\E T(x,y) \geq \E T_{F_+}(x,y)$. Choosing $x=0$ and $y=n\ebf_2$,
\[
\tfrac{1}{n} \E T(0,n\ebf_2) \;\geq\; \tfrac{1}{n} \E T_{F_+}(0,n\ebf_2)\ .
\]
Taking limits, $\bar{\mu}(\ebf_2)\ge \mu_+(\ebf_2)$, and this proves the theorem.
\end{proof}

\begin{remark}
The same argument applies in the case of a columnar defect, i.e., when edges along the vertical axis are assigned weights according to $F_0$ while remaining edges are assigned weights according to $F_-=F_+$. Assuming that $\E Y_\pm$ and $\E Y_0$ are finite, this shows that $\bar\mu(\ebf_2)=\mu_-(\ebf_2)=\mu_+(\ebf_2)$ as long as $F_-=F_+\prec F_0$.
\end{remark}

\begin{proof}[\bf Proof of Theorem~\ref{thm: strict_inequality}]
There are distributions such that in the second-coordinate direction $\bar{\mu}<\min\{\mu_-,\mu_+\}$, as seen in the following example. Let $F_+$ be an arbitrary non-degenerate distribution with $\mu_+=1$, and let $F_-$ be the degenerate distribution $\delta_1$. From \cite[Theorem~6.4]{kesten86}, we can find $y<\mu_+=1$ such that $F_+(y)>0$. Now choose any integer $K > 2y/(1-y)$ and note that 
\begin{equation}\label{eq: K_condition}
(K+2)y\,<\,K\ .
\end{equation}
Next, define the event $A$ that the edges in the set 
\[
\mathcal{K} = \big\{\{(-1,0),(0,0)\}, \{(-1,K),(0,K)\} \cup \{(-1,k),(-1,k+1)\} : k =0, \ldots, K-1\big\}
\]
have weight at most $y$. We now build a random path $\gamma$ from the origin along the $\mathbf{e}_2$ axis in a configuration $\{\tau_e\}_{e\in\Ec}$ as a concatenation of paths $(\gamma_i)_{i \geq 0}$ defined to have edge sets
\[
\mathcal{E}(\gamma_i) = \begin{cases}
T_{iK\mathbf{e}_2} \mathcal{K} & \text { if } T_{iK\mathbf{e}_2} A \text{ occurs}\ , \\
\{\{(0,iK+j),(0,iK+j+1)\} : j = 0, \ldots, K-1\} & \text{ otherwise}\ ,
\end{cases}
\]
where $T_{iK\mathbf{e}_2}$ is the operator that translates the point $iK\mathbf{e}_2$ to the origin. In words, the path $\gamma_i$ follows the edges in the translate $T_{iK\ebf_2}\mathcal{K}$ of $\mathcal{K}$ (from $(0,iK)$ to $(0,(i+1)K)$) if they all have weight at most $y$ and follows the $\mathbf{e}_2$-axis otherwise.

We can now compute the passage time of $\gamma$ up to its intersection with $(0,jK)$ as 
\[
T((\gamma_i)_{i\le j-1}) = K(j-N) + (K+2)Ny\ ,
\]
where $N$ is the random number of occurrences of translates of $A$:
\[
N := \sum_{i=0}^{j-1} \mathbf{1}_{T_{iK\mathbf{e}_2}A}(\{\tau_e\}_{e\in\Ec})\ .
\]
This is an upper bound for the minimal passage time, so
\[
\frac{T(0,jK\ebf_2)}{jK} \,\leq\, 1+ \frac{N}{j} \cdot \frac{(K+2)y-K}{K}\ .
\]
As $j \to \infty$ the left side converges to $\bar{\mu}(\mathbf{e}_2)$ almost surely. The law of large numbers implies that the right side converges almost surely to $1 + \mathbb{P}(A) \frac{(K+2)y-K}{K}$, which by~\eqref{eq: K_condition} is strictly less than 1. 

This does not yet prove Theorem~\ref{thm: strict_inequality}, since $F_-$ is degenerate, so we do a limiting argument. Let $F_-^{(m)}=(1-\frac{1}{m})\delta_1+\frac{1}{m}\delta_2$. Let $T$ denote passage time with respect to $F_-$ and $F_+$, and $T^{(m)}$ with respect to $F_-^{(m)}$ and $F_+$. Last, let $\bar{\mu}_m$ be the time constant for the time $T^{(m)}$. Choose $\eps>0$ and $n$ such that $\frac{1}{n}\E[T(0,n\ebf_2)]<1-2\eps$. By Monotone Convergence Theorem, for large $m$, $\frac{1}{n}\E[T^{(m)}(0,n\ebf_2)]<1-\eps$. Since $\bar{\mu}_m(\ebf_2)$ is obtained as an infimum over $n$, it is strictly less than 1. Thus the theorem holds with the pair $(F_+,F_-^{(m)})$, when $m$ is large.
\end{proof}

\begin{remark}
The above argument can be generalized in some simple ways. By scaling, the restriction $\mu_+=1$ can be removed. Further, $(F_-^{(m)})$ can be chosen as any sequence converging weakly to a delta mass, thus allowing for examples using continuous distributions and distributions with unbounded support.
\end{remark}

\section{Randomly introduced columnar defects}\label{sec: random_defects}

We finish by returning to the effect of columnar defects. We will show that for every $\eps>0$, if we introduce a defect independently with probability $\eps$ for each column $\{n\}\times\Z$, then the time constant will change in the vertical direction.

Equip $\{0,1\}^\Z$ with product measure $\Pbf_\eps$ whose marginal assigns weight $\eps\in(0,1)$ to $1$. Let $\xi=\{\xi_e\}_{e\in\Ec}$ be an independent family uniform on $[0,1]$, and let $F$ and $F_0$ be two distribution functions supported on $[0,\infty)$. As before, let the law for $\xi$ be denoted by $\Pr$. Given $\eta\in\{0,1\}^\Z$, assign to each vertical edge $e=\{(x,y),(x,y+1)\}$
\bea
\tau_e=\left\{
\begin{aligned}
&F^{-1}(\xi_e) & \text{if }\eta_x=0\ ,\\
&F^{-1}_0(\xi_e) & \text{if }\eta_x=1\ .
\end{aligned}
\right.
\eea
Each horizontal edge $e=\{(x,y),(x+1,y)\}$ is assigned weight $F^{-1}_0(\xi_e)$ in case $\eta_x=\eta_{x+1}=1$, and $F^{-1}(\xi_e)$ otherwise. We have thus created two layers of randomness, given by $\eta\sim\Pbf_\eps$ and $\xi\sim\Pr$, which respectively determines where columnar defects will occur, and the actual edge weights. Write $T_\eta$ for the passage time in the environment $(\eta,(\tau_e))$. We will prove the following.

\begin{thm}\label{thm: random_defects}
Let $\eps>0$, $F_0\prec F$, and assume that $\E Y_F$ and $\E Y_{F_0}$ are finite. Then
$$
\lim_{n\to\infty}\frac{T_\eta(0,n\ebf_2)}{n}=\mu_{F_0}(\ebf_2)\quad\text{with $(\mathbf{P}_\eps \!\times\! \Pr)$-probability one}\ .
$$
\end{thm}

Together with the characterization of~\cite[Theorem~2.9(b)]{vdbkes93}, Theorem~\ref{thm: random_defects} gives a weak criterion for randomly introduced columnar defects to result in a strictly smaller time constant: Let $\lambda=\inf\supp(F)$ and let $\vec{p}_c$ denote the critical probability for oriented bond percolation on $\Z^2$. If $F$ satisfies $F(0)<p_c$ and $F(\lambda)<\vec{p}_c$, then $\mu_{F_0}<\mu_F$ for any $F_0\neq F$ which is more variable than $F$.

\begin{proof}
Let $F$ and $F_0$ be given, and fix $\delta>0$. First note that the argument used to prove Theorem~\ref{thm: domination} also shows that, for every $\eta\in\{0,1\}^\Z$, $\liminf_{n\to\infty}T_\eta(0,n\ebf_2)/n\ge\mu_{F_0}$ $\Pr$-almost surely. So, it suffices to prove the remaining upper bound.

Let $T_{\eta,K}$ denote the restriction of $T_\eta$ to paths contained in the cylinder given by $|x|\le K$. The limit $\lim_{n\to\infty}\frac{1}{n}T_{\eta,K}(0,n\ebf_2)$ exists almost surely by the Subadditive Ergodic Theorem. Moreover, there is $K=K(\delta)$ such that for every $\eta$ with $\eta_x=1$ for all $|x|\leq K$, the limit is bounded by $\mu_{F_0}(\ebf_2)+\delta$ (see e.g.~\cite[Proposition~8]{A13-2}). In particular,
\be\label{eq: tube}
\limsup_{n\to\infty}\frac{T_\eta(0,n\ebf_2)}{n}\,\le\,\mu_{F_0}(\ebf_2)+\delta \quad\text{with $\Pr$-probability 1}\ .
\ee
By Borel-Cantelli, $\{\eta_{x+m}=1\text{ for all }|x| \leq K\}$ occurs for some $m\ge0$ for $\Pbf_\eps$-almost every $\eta$. Let $M<\infty$ denote the least positive integer $m$ for which it does. By subadditivity
$$
T_\eta(0,n\ebf_2)\,\le\, T_\eta(0,M\ebf_1)+T_\eta(M\ebf_1,M\ebf_1+n\ebf_2)+T_\eta(M\ebf_2+n\ebf_2,n\ebf_2)\ ,
$$
so division by $n$ and taking limits we obtain from~\eqref{eq: tube} that
$$
(\mathbf{P}_\eps \!\times\! \Pr)\Big(\limsup_{n\to\infty}\frac{T_\eta(0,n\ebf_2)}{n}\le \mu_{F_0}(\ebf_2)+\delta\Big)=1\ .
$$
Since $\delta>0$ was arbitrary, this concludes the proof.
\end{proof}

\begin{remark}
One may define various versions of the above defected model. For example, we could take all edge-weights for horizontal edges to be distributed as $F$ as well (with defects only present on vertical edges) and the limit in this case would be the time constant for the lattice with horizontal weights assigned from $F$ and vertical ones from $F_0$.
\end{remark}

\appendix

\section{Proof sketch for Proposition~\ref{prop: homogeneousLDE}}

We finish with an outline of the proof of Proposition~\ref{prop: homogeneousLDE}. The first inequality appeared in~\cite{A13} and the proof is a modification of the strategy of Grimmett and Kesten~\cite{grikes84,kesten86}, so we omit it. The proof of the second inequality is more involved and we sketch the proof here.

The original proof of Proposition~\ref{prop: homogeneousLDE} in~\cite{A13,A13-2} was for $x\in\Z^2$, and not for $x\in\R^2$. However, the latter easily follows from the former: let $z_x\in\Z^2$ be such that $x\in z_x+[0,1)^2$. Because $|\mu_F(x)-\mu_F(z_x)|\le2\mu_F(\ebf_1)$, for $t \geq \varepsilon/(2\mu_F(\mathbf{e}_1))$,
$$
\Pr\big(T_F(0,x)-\mu_F(x)>2\eps t\big)\le\Pr\big(T_F(0,z_x)-\mu_F(z_x)>\eps t\big)
$$
and we can apply the integer case explained below to the right side.

The goal of the rest of the appendix is to outline the proof of \eqref{eq: LD} using a regeneration argument from \cite{A13,A13-2}. Given $z\in\Z^2$ and $r>0$, let $\Cyl(z,r)$ denote the infinite cylinder $\bigcup_{a\in\R}\{x\in\R^2:|x-az|\le r\}$. We will assume throughout that $\E Y_+<\infty$ and assign i.i.d. passage times from $F_+$ to all edges in $(\mathbb{Z}^2,\mathcal{E})$.

\medskip
\noindent
{\bf Step 1.} \emph{Comparison between cylinder and full-space.} By subadditivity (Fekete's lemma), the cylinder time constant exists: 
\be\label{eq:fekete}
\mu_{\Cyl(z,r)}:=\lim_{n\to\infty}\frac{\E T_{\Cyl(z,r)}(0,nz)}{n}\ ,
\ee
where $T_{\Cyl(z,r)}(x,y)$ is the passage time over paths in $\Cyl(z,r)$. Furthermore, 
\be\label{eq:muKconv}
\lim_{r\to\infty}\mu_{\Cyl(z,r)}=\mu_+(z) \text{ for all } z \in \mathbb{Z}^2\ .
\ee
(See, for example, \cite[Proposition~8]{A13-2}.) So when the thickness of the cylinder increases, its time constant becomes arbitrarily close to the time constant for $F_+$.

\medskip
\noindent
{\bf Step 2.}  \emph{A regenerative approach.} We now compare $T_{\Cyl(z,r)}(0,nz)$ and the sum of travel times between randomly chosen `cross-sections' of $\Cyl(z,r)$. By symmetry we may assume $z = (z_1,z_2) \in\Z^2$ has $z_1,z_2\ge0$. Let $\Hc_n=\{y\in\Z^2: y_1+y_2=n\}$, $r\ge0$, and let $\tbar\in\R_+$ have $\Pr(\tau_e\le\tbar)>0$. Set
\bea
\begin{aligned}
E_n(z,r)&:=\{\{x,y\} \in \mathcal{E} : x\in \Cyl(z,r)\cap\Hc_{n\|z\|_1-1}, y \in \Cyl(z,r)\cap\Hc_{n\|z\|_1}\}\ ,\\
A_n(z,r)&:=\{\{\tau_e\}_{e\in\Ec} : \tau_e\le\tbar\text{ for all }e\in E_n(z,r)\}\ ,\\
\rho_j(z,r)&:=\min\{n>\rho_{j-1}(z,r):A_n(z,r)\text{ occurs}\}\text{ for }j\ge1,\quad \rho_0=0\ .
\end{aligned}
\eea
When clearly understood from the context, the reference to $z$ and $r$ will be dropped.

Note that $\{A_n(z,r)\}_{n\geq1}$ are i.i.d., so $\{\rho_j-\rho_{j-1}\}_{j\ge1}$ are independent geometrically distributed with success probability $\Pr(A_0(z,r))$ and $\{T_{\Cyl(z,r)}(\Hc_{\rho_{j-1}\|z\|_1},\Hc_{\rho_j\|z\|_1})\}_{j\ge1}$ are i.i.d. This sequence will serve as the increments in a renewal sequence. We note the following:

\begin{prop}\label{prop:regincrement}
If $\E Y_+<\infty$, then $\E\big[T_{\Cyl(z,r)}(\Hc_{\rho_0\|z\|_1},\Hc_{\rho_1\|z\|_1})^2\big]<\infty$ for all large $r$.
\end{prop}
\noindent
The proof is straightforward and involves bounding $T_{\Cyl(z,r)}$ by the minimum of passage times of a large number of disjoint paths of length $(\rho_1-\rho_0)\|z\|_1$.

\medskip
\noindent
{\bf Step 3.} \emph{Large deviations for cylinders.} Introduce the stopping time $\nu(m)=\nu(m,z,r):=\min\{j\ge1: \rho_j(z,r)>m\}$ and the notation
\bea
\mu_\tau(z,r):=\E[T_{\Cyl(z,r)}(\Hc_{\rho_0\|z\|_1},\Hc_{\rho_1\|z\|_1})]\quad\text{and}\quad\mu_\rho(z,r):=\E[\rho_1-\rho_0]\ .
\eea
Note that $\nu(m)-1$ equals the number of $n\in\{1,2,\ldots,m\}$ for which $A_n(z,r)$ occurs, and is binomially distributed with success probability $\Pr(A_0(z,r))=\mu_\rho(z,r)^{-1}$. It is straightforward to show that $\mu_{\Cyl(z,r)}$ is close to $\mu_{\tau}(z,r)/\mu_{\rho}(z,r)$.
The main step in the proof of Proposition~\ref{prop: homogeneousLDE} is:
\begin{prop}\label{prop:tube}
Assume that $\E Y_+<\infty$. For every $\eps>0$, $q\ge1$ and $z\in\N^2$, there exists $R_1=R_1(q)$ and $M_1=M_1(\eps,q,z)$ such that for every $r\ge R_1$, $n\in\N$ and $x\ge n$
 $$
 \Pr\big(T_{\Cyl(z,r)}(\Hc_0,\Hc_{n\|z\|_1})-n\mu_{\Cyl(z,r)}>\eps x\|z\|_1\big)\,\le\,\frac{M_1}{x^q}\ .
 $$
\end{prop}

\begin{proof}[Outline of proof]
The main tools are Chebychev's inequality and Wald's equation. Fix $\eps>0$ and choose $r$ large enough that $T_{\Cyl(z,r)}(\Hc_0,\Hc_{\|z\|_1})$ has finite variance (see Proposition~\ref{prop:regincrement}). Next choose $m$ large for $\frac{1}{m}\E\big[T_{\Cyl(z,r)}(\Hc_0,\Hc_{m\|z\|_1})\big]$ to be close to $\mu_{\Cyl(z,r)}$ (see~\eqref{eq:fekete}). Set $y=mz$.

$\big(T_{\Cyl(z,r)}(\Hc_0,\Hc_{\rho_j\|y\|_1})\big)_{j\ge1}$ is not a renewal sequence, but by the choice of cross-sections defining the $\rho_i$'s, it follows that $T_{\Cyl(z,r)}(\Hc_0,\Hc_{\rho_j\|y\|_1})$ is well-approximated by the sum of $T_{\Cyl(z,r)}(\Hc_{\rho_{i-1}\|y\|_1},\Hc_{\rho_i\|y\|_1})$ for $i=1,2,\ldots,j$ (where $\rho_0=0$). Note that these terms are i.i.d.\ by construction, and have finite variance. We will therefore use optimal stopping to approximate $T_{\Cyl(z,r)}(\Hc_0,\Hc_{n\|z\|_1})$ by a stopped sum. In particular,
\bea
\begin{aligned}
\Pr\Big(T_{\Cyl}(\Hc_0,\Hc_{\nu(n)\|y\|_1})-n\mu_+(y)>3x\Big)\;&\le\; \Pr\bigg(\sum_{i=1}^{\nu(n)}\Big(T_{\Cyl}(\Hc_{\rho_{i-1}\|y\|_1},\Hc_{\rho_i\|y\|_1})-\mu_\tau(y,r)\Big)>x\bigg)\\
&+\Pr\Big(\nu(n)\mu_\tau(y,r)-n\mu_+(y)>x\Big)+\,\text{small error}\ .
\end{aligned}
\eea

We apply Chebychev's inequality to the second term in the right side for the upper bound $M/x$. For the first term, we use both Chebychev's inequality and Wald's lemma for a similar bound $M/x$, and this implies the proposition for $q=1$. For larger $q$, one strengthens the bound by considering disjoint cylinders of radius $r$ aligned next to each other.
\end{proof}

\medskip
\noindent
{\bf Step 4.} \emph{Large deviations for the half-plane.} To derive the estimate for the half-plane from Proposition~\ref{prop:tube} we need to compare travel times on a cylinder with those in the half-plane. To circumvent the fact that the finite chunk $C_r$ of $\Cyl(z,r)$ in between $\Hc_0$ and $\Hc_{n\|z\|_1}$ may intersect the left half-plane we will shift space slightly. Namely, shifting $C_r$ by the vector $(r+1,r+1)$ completely includes it in the first quadrant, and 
$$
\Pr\big(T_+(\tilde\Hc_{2r+2},\tilde\Hc_{n\|z\|_1+2r+2})>x\big)\le\Pr\big(T_{\Cyl(z,r)}(0,nz)>x\big)\ ,
$$
where $\tilde\Hc_{n+2r+2}=\Hc_{n+2r+2}\cap[n,\infty)^2$. We now combine this with~\eqref{eq:muKconv} and Proposition~\ref{prop:tube}: for every $\eps>0$, $q\ge1$ and $z\in\Hb_+$, there exists $M_2=M_2(\eps,q,z)$ such that for every $n\in\N$ and $x\ge n$
\be\label{eq:half-plane1}
\Pr\big(T_+(2\ebf_1,2\ebf_1+nz)-n\mu_+(z)>2\eps x\|z\|_1\big)\,\le\,M_2\,\Pr(Y>x/M_2)+\frac{M_2}{x^q}\ .
\ee
The first term on the right is a bound on a sum of terms of the form $T_+(x,y)$, where $x$ and $y$ are at bounded distance. It is an error from `surgery' needed to connect $2\mathbf{e}_1$ to $\tilde\Hc_{2r+2}$, and $\tilde\Hc_{n\|z\|_1+2r+2}$ to $2\mathbf{e}_1+nz$. \eqref{eq:half-plane1} is close to the statement we aim to prove, and it remains only to remove the dependence of $M_2$ on the direction $z$. The argument is similar to the step from radial convergence of passage times to the shape theorem (as in the proof of \eqref{eq: inverted_shape}). One first obtains a control on large passage times of the type
\[
\Pr\big(T_+(2\ebf_2,2\ebf_2+z)>M_3x\big)\,\le\, M_3\,\Pr(Y>x/M_3)+\frac{M_3}{x^q}
\]
for all $z\in \Hb_+$ and $x\geq \|z\|_1$ and then argues as in \eqref{eq: inverted_shape}.

\bigskip
\noindent
{\bf Acknowledgements.} The authors would like to thank MSRI and U. C. Berkeley for its hospitality during a stay related to this work. 
The research of D.A. was supported by a post-doctoral scholarship from CNPq.
The research of M.D. is supported by NSF grant DMS-0901534.
The research of V.S. was supported in part by CNPq grants 308787/2011-0 and  476756/2012-0 and FAPERJ grant E-26/102.878/2012-BBP.
This work was also supported by ESF RGLIS Excellence Network.

\bibliographystyle{alpha}
\bibliography{bibpercolation}

\end{document}